\documentclass[a4paper,12pt]{elsarticle}%
\usepackage{amsmath}
\usepackage{amsfonts}
\usepackage{amssymb}
\usepackage{hyperref}
\usepackage{fourier}
\usepackage{hyphenat}
\usepackage{epic}
\usepackage{wrapfig}

\title{On $K_5$ and $K_{3,3}$-minors of graphs and regular matroids}
\author{Jo\~ao Paulo Costalonga}

\address{{\upshape joaocostalonga@yahoo.com.br}\\
  Universidade Estadual de Maring\'a\\
  Av. Colombo, 5790, Dpto. de Matem\'atica, Bl. F67, sl. 221\\
  Maring\'a-PR, 87020-900, Brazil\\
Tel. +55 44 3011 5356
	 }

\newenvironment{proofof}{}{\hfill$\Box$\medskip}
\newenvironment{proof}{\noindent{\it Proof: }}{\hfill$\Box$\medskip}

\newtheorem{lemma}{Lemma}[section]
\newtheorem{theorem}[lemma]{Theorem}
\newtheorem{corollary}[lemma]{Corollary}
\newtheorem{prop}[lemma]{Proposition}

\newcommand{\cont}{\subseteq}
\renewcommand{\u}{\cup}

\newcommand{\s}{^*}
\newcommand{\del}{\backslash}

\begin{document}
\begin{abstract}
In this paper we prove two main results about obstruction to graph planarity.
One is that, if $G$ is a $3$-connected graph with a $K_5$-minor and $T$ is a triangle of $G$, then $G$ has a $K_5$-minor $H$, such that $E(T)\cont E(H)$.

Other is that  if $G$ is a $3$-connected simple non-planar graph not isomorphic to $K_5$ and $e,f\in E(G)$, then $G$ has a minor $H$ such that $e,f\in E(H)$ and, up to isomorphisms, $H$ is one of the four non-isomorphic simple graphs obtained from $K_{3,3}$ by the addiction of  $\,0$, $1$ or $2$ edges. We generalize this second result to the class of the regular matroids.
\\
\end{abstract}


\maketitle

\newcommand{\kttvertices}{\put(10,0){\circle*{4}}\put(10,20){\circle*{4}}\put(10,40){\circle*{4}}\put(30,0){\circle*{4}}\put(30,20){\circle*{4}}\put(30,40){\circle*{4}}\put(-4,0){$u_3$}\put(-4,20){$u_2$}\put(-4,40){$u_1$}\put(32,0){$v_3$}\put(32,20){$v_2$}\put(32,40){$v_1$}}
\newcommand{\kttedges}{\drawline(10,0)(30,0)(10,20)(30,20)(10,40)(30,40)(10,0)(30,20)
\drawline(10,20)(30,40)\drawline(10,40)(30,0)}

\newcommand{\kttuuarestas}{\drawline(10,40)(30,20)(10,0)(30,0)(10,20)(30,20)(30,0)(10,40)(30,40)(10,0)(10,20)(30,40)}
\newcommand{\ktqvertices}{\put(10,10){\circle*{4}}\put(10,30){\circle*{4}}\put(10,50){\circle*{4}}\put(40,0){\circle*{4}}\put(40,20){\circle*{4}}\put(40,40){\circle*{4}}\put(40,60){\circle*{4}}}
\newcommand{\edgepata}{\drawline(40,60)(10,30)(10,10)(40,0)(10,30)(40,40)(40,20)(40,0)(10,50)(40,60)(10,10)}
\newcommand{\labelpata}{\put(-4,10){$u_3$}\put(-4,30){$u_2$}\put(-4,50){$u_1$}\put(42,-2){$v_3$}\put(42,18){$w_2$}\put(42,40){$w_1$}\put(42,60){$v_1$}}
\newcommand{\kttuu}{K_{3,3}^{1,1}}

\newcommand{\F}{\mathcal{F}}
\newcommand{\N}{\mathcal{N}}

\section{Introduction}

We use the terminology set by Oxley~\cite{Oxley}. Our graphs are allowed to have loops and multiple edges. When there is no ambiguity we denote by $uv$ the edge linking the vertices $u$ and $v$. We use the notation $si(G/e)$ for a simplification of $G$ (a graph obtained from $G$ by removing all loops, and all, but one, edges in each parallel class). Usually we choose the edge-set of $si(G)$ satisfying our purposes with no mentions. It is a consequence of Whitney's 2-Isomorphism Theorem (Theorem 5.3.1 of \cite{Oxley}) that, for each $3$-connected graphic matroid $M$, there is, up to isomorphisms, a unique graph whose $M$ is the cycle matroid. We also use this result without mention, so as Kuratowski and Wagner Theorems about graph planarity. When talking about a {\bf triangle} in a graph we may be referring both to the subgraph corresponding to the triangle as to its edge-set. We say that a set of vertices in a graph is {\bf stable} if such set has no pair of vertices linked by an edge.

Let $U$ and $V$ be different maximal stable sets of vertices in $K_{3,3}$. We define $K_{3,3}^{i,j}$ to be the simple graph obtained from $K_{3,3}$ by adding $i$ edges linking pairs of vertices of $U$ and $j$ edges linking pairs of vertices of $V$. By default, we label the vertices of $K_{3,3}^{i,j}$ like in Figure 1.
\begin{figure}\caption{Labels for extensions of $K_{3,3}$}
\begin{center}
\begin{picture}(40,50)(0,-15)
\kttvertices\kttedges\put(10,-15){$K_{3,3}$}
\end{picture}\hspace{0.4cm}
\begin{picture}(40,50)(0,-15)\put(10,-15){$K_{3,3}^{0,1}$}
\kttvertices\kttedges\drawline(30,0)(30,20)
\end{picture}\hspace{0.4cm}
\begin{picture}(40,50)(0,-15)\put(10,-15){$K_{3,3}^{0,2}$}
\kttvertices\kttedges\drawline(30,0)(30,20)(30,40)
\end{picture}\hspace{0.4cm}
\begin{picture}(40,50)(0,-15)\put(10,-15){$K_{3,3}^{1,1}$}
\kttvertices\kttedges\drawline(30,0)(30,20)\drawline(10,0)(10,20)
\end{picture}\hspace{0.4cm}
\begin{picture}(40,50)(0,-15)\put(10,-15){$K_{3,3}^{1,2}$}
\kttvertices\kttedges\drawline(30,0)(30,20)(30,40)\drawline(10,0)(10,20)
\end{picture}\hspace{0.4cm}
\begin{picture}(40,50)(0,-15)\put(10,-15){$K_{3,3}^{2,2}$}
\kttvertices\kttedges\drawline(30,0)(30,20)(30,40)\drawline(10,0)(10,20)(10,40)
\end{picture}
\end{center}
\end{figure}

A family $\F$ of matroids (graphs, resp.) is said to be {\bf $k$-rounded} in a minor-closed class of matroids (graphs, resp.) $\N$ if each member of $\F$ is $(k+1)$-connected and, for each $(k+1)$-connected matroid (graph, resp.) $M$ of $\N$ with an $\F$-minor and, for each $k$-subset $X\cont E(M)$, $M$ has an $\F$-minor with $X$ in its ground set (edge set, resp.). When $\N$ is omitted we consider it as the class of all matroids (graphs, resp.). By Whitney's 2-isomorphism Theorem, the concepts of $k$-roundedness for graphs and matroids agree, for $k\ge 2$. Such definition is a slight generalization of that one made by Seymour~\cite{Seymour1985}. For more information about $k$-roundedness we refer the reader to Section 12.3 of \cite{Oxley}.

The second main result stated in the abstract is Corollary \ref{sec-res}, that follows from the next Theorem we establish here:

\begin{theorem}\label{k33-classes} The following families of graphs are $2$-rounded:
\begin{enumerate}
\item [(a)]$\{K_{3,3},K_{3,3}^{0,1},K_{3,3}^{0,2},K_{3,3}^{1,1}\}$ and
\item [(b)]$\{K_{3,3},K_{3,3}^{0,1},K_{3,3}^{0,2},K_{3,3}^{1,1},K_5\}$.
\end{enumerate}
Moreover, the following families of matroids are $2$-rounded in the class of the regular matroids.
\begin{enumerate}
\item [(c)]$\{M(K_{3,3}),M(K_{3,3}^{0,1}),M(K_{3,3}^{0,2}),M(K_{3,3}^{1,1})\}$ and
\item [(d)]$\{M(K_{3,3}),M(K_{3,3}^{0,1}),M(K_{3,3}^{0,2}),M(K_{3,3}^{1,1}),M(K_5)\}$.
\end{enumerate}
\end{theorem}

Seymour~\cite[(7.5)]{Seymour1980}
proved that each $3$-connected simple non-planar graph not isomorphic to $K_5$ has a $K_{3,3}$-minor. So, as consequence of Theorem \ref{k33-classes} we have:

\begin{corollary}\label{sec-res}
If $G$ is a $3$-connected simple non-planar graph and $e,f\in E(G)$, then either $G\cong K_5$ or $G$ has a minor $H$ isomorphic to $K_{3,3},K_{3,3}^{0,1},K_{3,3}^{0,2}$ or $K_{3,3}^{1,1}$ such that $e,f\in E(H)$.
\end{corollary}

The next Corollary follows from Theorem \ref{k33-classes}, combined with Bixby's Theorem about decomposition of connected matroids into 2-sums (\cite[Theorem 8.3.1]{Oxley}). To derive the next corollary, instead of Theorem \ref{k33-classes}, we also may use a result of Seymour~\cite{Seymour1985}, which states that $\{U_{2,4},M(K_{3,3}),M(K_{3,3}^{0,1}),M(K_5)\}$ is $1$-rounded.

\begin{corollary}
If $G$ is a non-planar $2$-connected graph and $e\in E(G)$, then $G$ has a minor $H$ isomorphic to $K_5$, $K_{3,3}$ or $K_{3,3}^{0,1}$ such that $e\in E(H)$.
\end{corollary}

The first result we state at the abstract is Corollary \ref{k5}, that follows from the following theorem:

\begin{theorem}\label{kttuu}
If $G$ is a $3$-connected simple graph with a $K_{3,3}^{1,1}$-minor and $T$ is a triangle of $G$, then $G$ has a $\kttuu$-minor with $E(T)$ as edge-set of a triangle.
\end{theorem}
 
\begin{corollary}\label{k5}
If $G$ is a $3$-connected simple graph with a $K_{5}$-minor and $T$ is a triangle of $G$, then $G$ has a $K_5$-minor with $E(T)$ as edge-set of a triangle.
\end{corollary}

Other results about getting minors preserving a triangle were proved by Asano, Nishizeki and Seymour~\cite{Asano}. Truemper~\cite{Truemper} proved that if $G$ has a $K_{3,3}$-minor, and $e$, $f$ and $g$ are the edges of $G$ adjacent to a degree-$3$ vertex, then $G$ has a $K_{3,3}$-minor using $e$, $f$ and $g$.

We define a class $\F$ of $3$-connected matroids to be {\bf $(3,k,l)$-rounded} in $\N$ provided the following property holds: if $M$ is a $3$-connected matroid in $\N$ with an $\F$-minor, $X\cont E(M)$, $|X|=k$ and $r(X)\le l$, then $M$ has an $\F$-minor $N$ such that $X\cont E(N)$ and $N|X=M|X$.

Another formulation for Theorem \ref{kttuu} and Corollary \ref{k5} is that $\{M(K_{3,3}^{1,1})\}$ and $\{M(K_5)\}$ are $(3,3,2)$-rounded in the class of graphic matroids. Costalonga~\cite{Costalonga-Vertically}(in the last comments of the introduction) proved:

\begin{prop}\label{costalonga}Let $2\le l\le k\le 3$. Let $\F$ be a finite family of matroids and $\N$ a class of matroids closed under minors.
Then, there is a $(3,k,l)$-rounded family of matroids $\F'$ such that each $M\in \F'$ has an $\F$-minor $N$ satisfying $r(M)-r(N)\le k+\lfloor\frac{k-1}{2}\rfloor$.
\end{prop}

In \cite{Costalonga-Vertically} there are more results of such nature. Although a minimal $(3,3,3)$-rounded family of graphs containing $\{K_5, K_{3,3}\}$ exists and even has a size that allows a computer approach, it has shown to be complicated. Such family must at least include the graphs $K_{3,3}^{i,j}$, for $i+j\le 3$, $K_5$ and the following two graphs in Figure 2, obtained, respectively, from $K_{3,3}$ and $K_{5}$ by the same kind of vertex expansion, which shall occur in such kind of families.
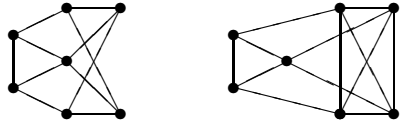
\begin{figure}\caption{Members of a $(3,3,3)$-rounded family containing $\{K_5, K_{3,3}\}$}
\begin{center}
\begin{picture}(40,40)
\put(0,10){\circle*{4}}\put(20,,0){\circle*{4}}\put(20,20){\circle*{4}}\put(20,40){\circle*{4}}
\put(0,30){\circle*{4}}\put(40,0){\circle*{4}}\put(40,40){\circle*{4}}
\drawline(0,10)(0,30)(20,40)(40,40)(20,20)(40,0)(20,0)(0,10)(20,20)(0,30)\drawline(40,40)(20,0)\drawline(20,40)(40,0)
\end{picture}\qquad\qquad
\begin{picture}(70,40)
\drawline(0,30)(0,10)(40,0)(60,0)(60,40)(40,40)(40,0)(60,40)(20,20)(60,0)(40,40)(0,30)(20,20)(0,10)
\put(0,30){\circle*{4}}\put(0,10){\circle*{4}}\put(40,0){\circle*{4}}\put(60,0){\circle*{4}}
\put(60,40){\circle*{4}}\put(40,40){\circle*{4}}\put(40,0){\circle*{4}}\put(20,20){\circle*{4}}
\end{picture}
\end{center}
\end{figure}
\section{Proofs for the Theorems}
The proof of Theorem \ref{k33-classes} is based on the following theorem:

\begin{theorem}\label{criterion}(Seymour~\cite{Seymour1985}, see also \cite[Theorem 12.3.9]{Oxley}) Let $\N$ be a class of matroids closed under minors, and $\F$ be a family of $3$-connected matroids. If, for each matroid $M$, for each $e\in E(M)$ such that $M/e\in\F$ or $M\del e \in \F$ and for each $f\in E(M)-e$ there is an $\F$-minor using $e$ and $f$, then $\F$ is $2$-rounded in $\N$.
\end{theorem}

Seymour proved Theorem \ref{criterion} when $\N$ is the class of all matroids. But the same proof holds for this more general version. By Whitney's 2-isomorphism Theorem, the analogous for graphs of Theorem \ref{criterion} holds. \\

\begin{proofof}\emph{Proof of theorem \ref{k33-classes}: }
For items (a) and (b) we will consider $\N$ as the class of graphic matroids and for items (c) and (d) we will consider $\N$ as the class of regular matroids. In each item we will verify the criterion given by Theorem \ref{criterion}.

First we prove item (a). We begin looking at the $3$-connected simple graphs $G$ such that $G\del e \in \F_a:=\{K_{3,3},K_{3,3}^{0,1},K_{3,3}^{0,2},K_{3,3}^{1,1}\}$. We may assume   that $G\notin \F_a$. So, up to isomorphisms, $G=K_{3,3}^{0,3}$ or $G=K_{3,3}^{1,k}$ for some $k\in \{1,2,3\}$. Thus $e\notin E(K_{3,3})$. Define $H:=G[E(K_{3,3})\u \{e,f\}]$. If $f\in E(K_{3,3})$, then $H\cong K_{3,3}^{0,1}$, otherwise $H\cong K_{3,3}^{0,2}$ or $H\cong K_{3,3}^{1,1}$. Thus $H$ is an $\F_a$-minor of $G$ and we may suppose that $G/e\in \F_a$.

e have that $G$ is $3$-connected and simple, in particular, $G$ has no degree-$2$ vertices, hence $G$ must be obtained from $G/e$ by the expansion of a vertex with degree at least $4$. This implies that $G\ncong K_{3,3}$. Thus, we may assume that $G/e$ is one of graphs $K_{3,3}^{0,1}$, $K_{3,3}^{1,1}$ or $K_{3,3}^{0,2}$. We denote $e:=w_1w_2$.
\begin{figure}
\begin{center}\caption{}
\begin{picture}(40,60)
\ktqvertices\labelpata
\drawline(40,60)(10,10)(40,0)(40,20)(40,40)(10,30)(40,0)(10,50)(40,40)
\drawline(10,30)(40,60)(10,50)
\drawline(10,10)(40,20)
\put(22,-12){$G_1$}
\end{picture}\hspace{1.5cm}
\begin{picture}(40,60)
\ktqvertices\labelpata
\drawline(40,60)(10,10)(40,0)(40,20)(40,40)(10,30)(40,0)(10,50)(40,40)
\drawline(10,30)(40,60)(10,50)
\drawline(10,10)(40,20)\drawline(40,40)(40,60)
\put(22,-12){$G_2$}
\end{picture}\hspace{1.5cm}
\begin{picture}(40,75)
\ktqvertices\labelpata
\drawline(40,60)(10,10)(40,0)(40,20)(40,40)(10,30)(40,0)(10,50)(40,40)
\drawline(10,30)(40,60)(10,50)
\qbezier(40,20)(75,50)(40,60)
\drawline(10,10)(40,40)
\put(22,-12){$G_3$}
\end{picture}\hspace{1.5cm}
\begin{picture}(40,75)
\ktqvertices\labelpata
\drawline(40,60)(10,10)(40,0)(40,20)(40,40)(10,30)(40,0)(10,50)(40,40)
\drawline(10,30)(40,60)(10,50)
\qbezier(40,20)(75,50)(40,60)
\drawline(10,10)(40,20)
\put(22,-12){$G_4$}
\end{picture}
\end{center}
\end{figure}

If $G/e= K_{3,3}^{0,1}$, then $G$ is obtained from $G/e$ by the expansion of a degree-~$4$ vertex. In this case we may assume without losing generality that $G$ is the graph $G_1$, defined in Figure 3. Note that, in this case, $G_1/u_3w_2\cong K_{3,3}$ and that $G_1/u_3v_1\cong K_{3,3}^{0,1}$(with $\{u_1,u_2,w_2\}$ stable). So, one of $G_1/u_3w_2$ or $G_1/u_3v_1$ is an $\F_a$-minor we are looking for. So we may assume that $G\neq K_{3,3}^{0,1}$.

If $G\cong K_{3,3}^{1,1}$, then $G\cong G_1+ u_2u_3$ and the result follows as in the preceding case. Hence we may assume that $G/e\cong K_{3,3}^{0,2}$.

If $G$ is obtained from $G/e$ by the expansion of a degree-$4$ vertex, then $G\cong G_2\cong G_1+v_1w_1$. In this case we may proceed as in the first case again.

Thus, if $G/e=K_{3,3}^{0,2}$, we can assume that $G$ is obtained from $G/e$ by the expansion of the degree-$5$ vertex. If $\{v_1w_1,v_3w_2\}$ or $\{v_1w_2,v_3w_1\}$ is contained in $E(G)$, then $G$ is again isomorphic to $G_2$ and we are reduced to the first case again. Without loss of generality, say that $v_1w_2,v_2w_2\in E(G)$. Then $G$ is one of the graphs $G_3$ or $G_4$ in Figure 3. If $G=G_3$, then one of $G_3/v_1w_2$ or $G_3/w_2v_3$, both isomorphic to $K_{3,3}^{0,2}$ is the $\F_a$-minor we are looking for. If $G=G_4$, then one of $si(G_4/u_3w_2)$ ($\cong K_{3,3}^{0,1}$) or $si(G_4/u_3v_1)$($\cong K_{3,3}^{0,1}$, with $\{u_1,u_2,w_2\}$ stable) is such an $\F_a$-minor. This proves item (a).

Now we prove item (b). We just have to examine the $3$-connected simple single-element extensions and coextensions of $K_5$, since other verifications were made in the proof of item (a). The unique graph $G$ with an edge $e$ such that $G/e\cong K_5$ or $G\del e \cong K_5$ is $K_{3,3}^{1,1}$ (up to isomorphisms). So, we have item (b).

Now we prove item (c). By the proof of item (a), it is just left to examine the $3$-connected extensions and coextensios of the matroids in $\F_c:=\{M(K_{3,3}),$ $M(K_{3,3}^{0,1}),$ $M(K_{3,3}^{0,2}),$ $M(K_{3,3}^{1,1})\}$ which are not graphic. By \cite[Theorem 13.1.2 and Proposition 12.2.8]{Oxley}, each $3$-connected regular matroid is graphic, cographic, isomorphic to $R_{10}$ or has a $R_{12}$-minor. But no cographic matroid has a minor in $\F_c$. Moreover, by cardinality, $R_{10}$ also has no $\F_c$-minor. So, the unique non-graphic matroids $M$ such that $M\del e$ or $M/e$ is possibly in $\F_c$ are those with $R_{12}$-minors. Specifically, by cardinality and rank the unique non-graphic matroid that possibly have a single element deletion or contraction in $\F_c$ is $R_{12}$, up to isomorphisms. Usually $R_{12}$ is defined as the matroid represented over $GF(2)$ by the following matrix:
$$B:=\bordermatrix{
&1&2&3&4&5&6&7&8&9&10&11&12\cr
z_1&1&0&0&0&0&0&1&1&1&0&0&0\cr
z_2&0&1&0&0&0&0&1&1&0&1&0&0\cr
z_3&0&0&1&0&0&0&1&0&0&0&1&0\cr
z_4&0&0&0&1&0&0&0&1&0&0&0&1\cr
z_5&0&0&0&0&1&0&0&0&1&0&1&1\cr
z_6&0&0&0&0&0&1&0&0&0&1&1&1}$$
Now, we build a representation of $si(R_{12}/1)$ as follows. First, we eliminate the first row and column of $B$ and eliminate column $9$, that became equal column $5$, after that, we add row $z_5$ to row $z_6$ and, finally, we add an extra row $z_7$ equal to the sum of the other rows. So we get the matrix $A$, defined next:
\begin{wrapfigure}[6]{r}[-30pt]{3cm}\caption{}\medskip
\begin{picture}(50,90)
\put(15,5){\circle*{4}}\put(0,5){$z_6$}\put(15,45){\circle*{4}}\put(0,45){$z_4$}
\put(15,85){\circle*{4}}\put(0,85){$z_3$}\put(70,5){\circle*{4}}\put(75,5){$z_ 5$}
\put(70,45){\circle*{4}}\put(75,45){$z_7$}\put(70,85){\circle*{4}}\put(75,85){$z_2$}
\drawline(70,85)(70,45)\put(70,65){2}\drawline(15,85)(70,45)\put(33,72){3}
\drawline(15,45)(70,45)\put(27,44){4}\drawline(15,5)(70,5)\put(41,-3){5}
\drawline(15,5)(70,45)\put(33,11){6}\drawline(15,85)(70,85)\put(41,85){7}
\drawline(15,45)(70,85)\put(18,51){8}\drawline(15,5)(70,85)\put(13,17){10}
\drawline(15,85)(70,5)\put(12,68){11}\drawline(15,45)(70,5)\put(13,31){12}
\end{picture}
\end{wrapfigure}
$$A:=\bordermatrix{
     &2&3&4&5&6&7&8&10&11&12\cr
z_2&1&0&0&0&0&1&1&1&0&0\cr
z_3&0&1&0&0&0&1&0&0&1&0\cr
z_4&0&0&1&0&0&0&1&0&0&1\cr
z_5&0&0&0&1&0&0&0&0&1&1\cr
z_6&0&0&0&1&1&0&0&1&0&0\cr
z_7&1&1&1&0&1&0&0&0&0&0
}$$
\medskip\\

Note that $ R_{12}/1\del 9\cong si(R_{12}/1)\cong R_{12}/1\del 5$ is the cycle matroid of a graph in Figure 4. Now, observe that, inverting the order of the rows in matrix $B$ give us an automorphism $\phi$ of $R_{12}$ such that $\phi(1)=6$. Moreover $R_{12}$ is self dual, where an isomorphism between $R_{12}$ and $R_{12}\s$ takes $1$ into $7$. So $\{1,6,7\}$ is in a orbit of the automorphism group of $R_{12}$. Thus $so(R_{12}/1)\cong si(R_{12}/6)\cong si(R_{12}/7)\cong M(K_{3,3}^{0,2})$ and the ground set of one of these matroids can be chosen containing $\{e,f\}$. Therefore, for each pair of elements of $R_{12}$, there is an $\F_c$-minor containing both. This proves item (c).

To prove item (d) we observe that if $M/e=M(K_5)$ or $M\del e \cong (K_5)$, then $|E(M)|=11$, so $M$ is not isomorphic to $R_{10}$ neither has an $R_{12}$-minor. Moreover $M$ is not cographic in this case. So, all matroids we have to deal are graphic, and the proof of item (d) is reduced to item (b).
\end{proofof}

\begin{lemma}\label{equivalent-minors}Let $G$ be a $3$-connected simple graph not isomorphic to $K_5$. Then $G$ has a $K_5$-minor if and only if $G$ has a $\kttuu$-minor.
\end{lemma}
\begin{proof}
If $G$ has a $\kttuu$-minor, then $G$ has a $K_5$-minor, because $K_5\cong K_{3,3}^{1,1}/u_1v_1$. In other hand, suppose that $G$ has a $K_5$-minor. By the Splitter Theorem (Theorem 12.1.2 of \cite{Oxley}), $G$ has a $3$-connected simple minor $H$ with an edge $e$ such that $H/e\cong K_5$ or $H\del e\cong K_5$. But no simple graph $H$ has an edge $e$ such that $H\del e\cong K_5$. So $H/e\cong K_5$. Now, it is easy to verify that $H\cong\kttuu$ and conclude the lemma.
\end{proof}

The next result is Corollary 1.8 of \cite{Costalonga2}.
\begin{corollary}\label{costalonga}
Let $G$ be a simple $3$-connected graph with a simple $3$-connected minor $H$ such that $|V(G)|-|V(H)|\geq3$. Then there is a $3$-subset $\{x,y,z\}$ of $E(G)$, which is not the edge-set of a triangle of $G$, such that $G/x$, $G/y$, $G/z$ and $G/x,y$ are all $3$-connected graphs having $H$-minors.
\end{corollary}

\begin{proofof}\emph{Proof of Theorem \ref{kttuu}: }
Suppose that $G$ and $T$ is a counter-example to the theorem minimizing $|V(G)|$. If $|V(G)|\ge 8$, by Corollary \ref{costalonga} applied to $G$ and $K_5$, $G$ has an edge $e$ such that, $e\notin cl_{M(G)}(T)$ and $G/e$ is $3$-connected and have a $K_5$-minor. Thus $si(G/e)$ is a $3$-connected simple graph having $T$ as triangle. By Lemma \ref{equivalent-minors}, $si(G/e)$ has a $\kttuu$-minor, contradicting the minimality of $G$. Thus $|V(G)|\le 7$. If $|V(G)|=6$, then $G\cong K_{3,3}^{i,j}$ for some $1\le i\le j \le 3$. In this case, the Theorem can be verified directly. Thus $|V(G)|=7$.

So, there is $e\in E(G)$ and $X\cont E(G)$ such that $G\del X/e \cong K_{3,3}^{1,1}$. If $e\notin T$, $si(G/e)$ contradicts the minimality of $T$, so $e\in T$. We split the proof into two cases now.

The first case is when $e$ is adjacent to a degree-$2$ vertex $v$ of $G\del X$. Let $f$ be the other edge adjacent to $v$ in $G\del X$. So $e,f\in T$, otherwise, $si(G/f)$ would contradict the minimality of $G$.  

Up to isomorphisms, $G\del X$ can be obtained from $K_{3,3}^{1,1}\cong G\del X/e$ by adding the vertex $v$ in the middle of some edge $e'$. By symmetry, we may assume that $e'\in\{u_1v_2,v_2v_3,u_1v_1\}$. So, there are, up to isomorphisms, three possibilities for $G\del (X-T)$, those in Figure 5.
\begin{figure}
\begin{center}\caption{}
\begin{picture}(50,50)
\kttvertices\kttuuarestas
\put(50,50){\circle*{4}}\put(51,48){$v$}
\drawline(10,40)(50,50)(30,20)
\put(17,-12){$G_1$}
\end{picture}\hspace{1cm}
\begin{picture}(50,50)
\kttvertices\kttuuarestas
\put(50,10){\circle*{4}}\put(50,10){$v$}
\drawline(30,0)(50,10)(30,20)
\put(17,-12){$G_2$}
\end{picture}\hspace{1cm}
\begin{picture}(50,50)
\kttvertices\kttuuarestas
\put(20,50){\circle*{4}}\put(21,50){$v$}
\drawline(10,40)(20,50)(30,40)
\put(17,-12){$G_3$}
\end{picture}
\end{center}
\end{figure}
Since $G$ is simple, $G$ has a third edge $g$ adjacent to $v$. For any of the graphs in Figure 5, it verifies that $si(G\del(X-T)/g)$ contradicts the minimality of $G$. So the proof is done in the first case.

In the second case, $e$ is an edge of $G\del X$ whose adjacent vertices has degree at least $3$. We may suppose that the end-vertices $w_1$ and $w_2$ of $e$ collapses into $v_2$ when contracting $e$ in $G\del X$.  Let $S$ be the set of edges incident to $v_2$ in $G\del X/e$. We also may assume that $w_2$ is adjacent to $v_3$ in $G\del X$. With this assumptions $G\del (X\u S)$ is the graph $G_4$ of Figure 6. Note also that $G\del X$ is obtained from $G_4$ adding $3$ edges, each incident to a different vertex in $\{u_1,u_2,u_3\}$, two of then incident to $w_1$ and one incident to $w_2$. Since switching $u_2$ and $u_3$ in $G_4$ induces an automorphism, we may suppose that $u_2w_1\in E(G\del X)$. Then, without losing generality, $G\del X$ is one of the graphs $G_5$ or $G_6$ in Figure 6.

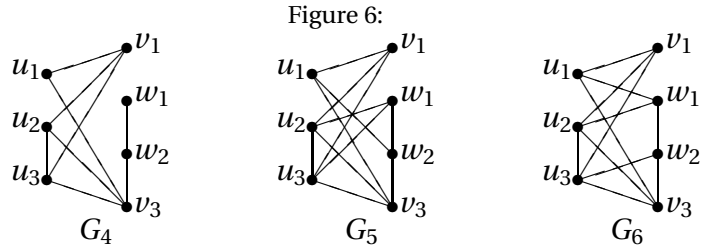
\begin{figure}[h]
\begin{center}\caption{}\medskip
\begin{picture}(40,60)
\ktqvertices\labelpata
\drawline(40,60)(10,30)(10,10)(40,0)(10,30)
\drawline(40,40)(40,20)(40,0)(10,50)(40,60)(10,10)
\put(22,-12){$G_4$}
\end{picture}\hspace{2cm}
\begin{picture}(40,60)
\ktqvertices\edgepata\labelpata
\drawline(10,50)(40,20)
\drawline(10,10)(40,40)
\put(22,-12){$G_5$}
\end{picture}\hspace{2cm}
\begin{picture}(40,60)
\ktqvertices\edgepata\labelpata
\drawline(10,50)(40,40)
\drawline(10,10)(40,20)
\put(22,-12){$G_6$}
\end{picture}
\end{center}
\end{figure}

In the case that $G=G_5$, in Figure 7, in the first row, for each possibility for $T$ we draw $G\del(X-T)$. The bold edges are those of $T$. In each graph of the first row, the double edge $g$ has the property that the graph $si(G\del(X-T)/g)$, draw in the second row in the respective column, contradicts the minimality of $G$. The vertex obtained in the contraction is labelled by $z$. In the third and fourth rows of Figure 7, we have the same for the case in which $G=G_6$. This proves the theorem
\end{proofof}
 
\newcommand{\kttright}{
\put(15,0){\circle*{4}}
\put(15,20){\circle*{4}}
\put(15,40){\circle*{4}}
\put(40,0) {\circle*{4}}
\put(40,20){\circle*{4}}
\put(40,40){\circle*{4}}
\drawline(15,0)(40,0)
\drawline(15,0)(40,20)
\drawline(15,0)(40,40)
\drawline(15,20)(40,0)
\drawline(15,20)(40,20)
\drawline(15,20)(40,40)
\drawline(15,40)(40,0)
\drawline(15,40)(40,20)
\drawline(15,40)(40,40)
}
\newcommand{\gfivebase}{\put(15,10){\circle*{4}}\put(0,8){$u_3$}
\put(15,30){\circle*{4}}\put(0,28){$u_2$}
\put(15,50){\circle*{4}}\put(0,48){$u_1$}
\put(45,0){\circle*{4}}\put(46,-2){$v_3$}
\put(45,20){\circle*{4}}\put(46,18){$w_2$}
\put(45,40){\circle*{4}}\put(46,38){$w_1$}
\put(45,60){\circle*{4}}\put(46,58){$v_1$} }
\newcommand{\boldtriangle}[3]{{\linethickness{1.2pt}
\qbezier#1#2#2
\qbezier#2#3#3
\qbezier#3#1#1}}
\begin{figure}\begin{center}\caption{}\medskip
\begin{picture}(50,60)
{\linethickness{1.2pt}
\qbezier(45,20)(45,30)(45,40)
\qbezier(45,40)(45,50)(45,60)
\qbezier(45,20)(77,43)(45,60)}
\gfivebase
\thinlines
\drawline(15,51)(45,61)(45,59)(15,49)
\drawline(15,50)(45,20)\drawline(15,50)(45,0)\drawline(15,30)(45,60)
\drawline(15,30)(45,40)\drawline(15,30)(45,0)\drawline(15,10)(45,60)
\drawline(15,10)(45,40)\drawline(15,10)(45,0)\drawline(45,40)(45,20)
\drawline(45,20)(45,0)
\drawline(15,10)(15,30)
\end{picture}\qquad
\begin{picture}(50,60)\gfivebase
{\linethickness{1.2pt}
\qbezier(15,50)(45,20)(45,20)\qbezier(45,20)(45,40)(45,40)\qbezier(45,40)(15,50)(15,50)
}
\thinlines
\drawline(15,50)(45,60)\drawline(15,50)(45,20)\drawline(15,50)(45,0)
\drawline(15,30)(45,60)\drawline(15,30)(45,40)\drawline(15,30)(45,0)
\drawline(15,10)(45,60)\drawline(15,10)(45,40)\drawline(15,10)(45,0)
\drawline(45,40)(45,20)
\drawline(46,20)(46,0)(44,0)(44,20)
\drawline(15,10)(15,30)\end{picture}\qquad
\begin{picture}(50,60)\gfivebase
\boldtriangle{(45,20)}{(45,40)}{(15,30)}
\thinlines
\drawline(15,50)(45,60)\drawline(15,50)(45,20)\drawline(15,50)(45,0)
\drawline(15,30)(45,60)\drawline(15,30)(45,40)\drawline(15,30)(45,0)
\drawline(15,10)(45,60)\drawline(15,10)(45,0)
\drawline(45,40)(45,20)\drawline(45,20)(45,0)\drawline(15,10)(15,30)
\drawline(16,10)(46,40)(44,40)(14,10)
\end{picture}\qquad
\begin{picture}(50,60)\gfivebase
\boldtriangle{(45,20)}{(45,40)}{(15,10)}
\thinlines
\drawline(15,50)(45,60)\drawline(15,50)(45,20)\drawline(15,50)(45,0)
\drawline(15,30)(45,60)\drawline(15,30)(45,0)
\drawline(15,10)(45,60)\drawline(15,10)(45,40)\drawline(15,10)(45,0)
\drawline(45,40)(45,20)\drawline(45,20)(45,0)\drawline(15,10)(15,30)
\drawline(15,29)(45,39)(45,41)(15,31)
\end{picture}\qquad
\begin{picture}(50,60)
{\linethickness{1.2pt}
\qbezier(45,40)(45,20)(45,20)\qbezier(45,20)(45,0)(45,0)\qbezier(45,40)(77,20)(45,0)}\thinlines\gfivebase
\drawline(15,50)(45,60)
\drawline(15,51)(45,21)(45,19)(15,49)
\drawline(15,50)(45,0)\drawline(15,30)(45,60)\drawline(15,30)(45,40)
\drawline(15,30)(45,0)\drawline(15,10)(45,60)\drawline(15,10)(45,40)
\drawline(15,10)(45,0)\drawline(45,40)(45,20)\drawline(45,20)(45,0)
\drawline(15,10)(15,30)\end{picture}

\begin{picture}(50,60)
\boldtriangle{(15,40)}{(40,40)}{(40,20)}\thinlines\drawline(15,20)(15,0)\kttright
\put(-1,38){$w_2$}\put(0,18){$u_2$}\put(0,-2){$u_3$}\put(42,38){$z$}\put(42,18){$w_1$}\put(42,-2) {$v_3$}
\end{picture}\qquad
\begin{picture}(50,60)
\boldtriangle{(15,40)}{(40,20)}{(40,0)}\thinlines\kttright\drawline(15,20)(15,0)
\put(0,38){$u_1$}\put(0,18){$u_2$}\put(0,-2){$u_3$}\put(41,38){$v_1$}\put(41,18){$w_1$}\put(42,-2) {$z$}
\end{picture}\qquad
\begin{picture}(50,60)
\boldtriangle{(15,0)}{(15,20)}{(40,20)}\thinlines\kttright\drawline(40,0)(40,20)
\put(0,38){$u_1$}\put(0,18){$u_2$}\put(0,-2){$\,\,\,z$}\put(41,38){$v_1$}\put(41,18){$w_2$}\put(41,-2) {$v_3$}
\end{picture}\qquad
\begin{picture}(50,60)
\boldtriangle{(15,0)}{(15,20)}{(40,20)}\kttright\drawline(40,0)(40,20)
\put(0,38){$u_1$}\put(0,18){$z$}\put(0,-2){$u_3$}\put(41,38){$v_1$}\put(41,18){$w_2$}\put(41,-2) {$v_3$}
\end{picture}\qquad
\begin{picture}(50,60)
\boldtriangle{(15,40)}{(40,0)}{(40,20)}
\thinlines\kttright\drawline(15,0)(15,20)
\put(0,38){$w_2$}\put(0,18){$u_2$}\put(0,-2){$u_3$}\put(41,38){$z$}\put(41,18){$w_1$}\put(41,-2) {$v_3$}
\end{picture}
\end{center}

\begin{center}
\newcommand{\twostars}{
\drawline(15,50)(45,60)(15,30)
\drawline(45,60)(15,10)(45,0)
\drawline(15,50)(45,0)(15,30)}
\begin{picture}(50,60)
\gfivebase
{\linethickness{1.2pt}\qbezier(45,20)(45,40)(45,60) \qbezier(45,20)(78,40)(45,60)}
\drawline(15,9)(45,19)(45,21)(15,11)
\thinlines\twostars\drawline(15,50)(45,40)(15,30)
\drawline(45,0)(45,20)\drawline(15,10)(15,30)
\end{picture}\qquad
\begin{picture}(50,60)
\gfivebase\boldtriangle{(15,50)}{(45,40)}{(45,20)}\thinlines
\twostars
\drawline(15,29)(45,39)(45,41)(15,31)
\drawline(45,0)(45,20)(15,10) \drawline(15,10)(15,30)
\end{picture}\qquad
\begin{picture}(50,60)
\gfivebase\boldtriangle{(15,30)}{(45,40)}{(45,20)}
\drawline(15,49)(45,39)(45,41)(15,51)
\thinlines\twostars\drawline(45,0)(45,20)(15,10)\drawline(15,10)(15,30)
\end{picture}\qquad
\begin{picture}(50,60)
\gfivebase\boldtriangle{(45,20)}{(45,40)}{(15,10)}
\drawline(44,20)(44,0)(46,0)(46,20)
\thinlines\twostars\drawline(15,50)(45,40)(15,30)\drawline(15,10)(15,30)
\end{picture}\qquad
\begin{picture}(50,60)
\gfivebase\boldtriangle{(45,0)}{(45,20)}{(45,40)}
{\linethickness{1.2pt} \qbezier(45,0)(75,20)(45,40)}
\drawline(15,29)(45,39)(45,41)(15,31)
\thinlines\twostars\drawline(15,50)(45,40)\drawline(15,10)(45,20)\drawline(15,10)(15,30)
\end{picture}

\begin{picture}(50,60)
\boldtriangle{(40,40)}{(40,20)}{(15,0)}
\thinlines\kttright\drawline(15,0)(15,20)
\put(0,38){$u_1$}
\put(0,18){$u_2$}
\put(0,-2){$z$}
\put(41,38){$v_1$}
\put(41,18){$w_1$}
\put(41,-2) {$v_3$}
\end{picture}\qquad
\begin{picture}(50,60)
\boldtriangle{(15,40)}{(40,20)}{(15,20)}
\thinlines\kttright\drawline(15,0)(15,20)\drawline(40,0)(40,20)
\put(0,38){$u_1$}
\put(0,18){$z$}
\put(0,-2){$u_3$}
\put(41,38){$v_1$}
\put(41,18){$w_2$}
\put(41,-2) {$v_3$}
\end{picture}\qquad
\begin{picture}(50,60)
\boldtriangle{(15,40)}{(40,20)}{(15,20)}
\thinlines\kttright\drawline(15,0)(15,20)\drawline(40,0)(40,20)
\put(0,38){$z$}
\put(0,18){$u_2$}
\put(0,-2){$u_3$}
\put(41,38){$v_1$}
\put(41,18){$w_2$}
\put(41,-2) {$v_3$}
\end{picture}\qquad
\begin{picture}(50,60)
\boldtriangle{(15,0)}{(40,20)}{(40,0)}
\thinlines\kttright\drawline(15,0)(15,20)\drawline(40,0)(40,20)
\put(0,38){$u_1$}
\put(0,18){$u_2$}
\put(0,-2){$u_3$}
\put(41,38){$v_1$}
\put(41,18){$w_2$}
\put(41,-2) {$z$}
\end{picture}\qquad
\begin{picture}(50,60)
\boldtriangle{(15,20)}{(40,20)}{(40,0)}
\thinlines\kttright\drawline(15,0)(15,20)\drawline(40,0)(40,20)
\put(0,38){$u_1$}
\put(0,18){$z$}
\put(0,-2){$u_3$}
\put(41,38){$v_1$}
\put(41,18){$w_2$}
\put(41,-2) {$v_3$}
\end{picture}
\end{center}
\end{figure}

\begin{proofof}\emph{ Proof of Theorem \ref{k5}: }
Suppose that $G$ is a $3$-connected simple graph with a $K_5$-minor and $T$ is a triangle of $G$. We may suppose that $G\ncong K_5$. By Lemma \ref{equivalent-minors}, $G$ has a $3$-connected simple minor $H\cong\kttuu$. By Theorem \ref{kttuu}, we choose $H$ having the edges of $T$ in a triangle. Let $e\in H$ be the edge such that $H/e\cong K_5$. Note that $e$ is in no triangle of $H$. So $H/e$ is the $K_5$-minor we are looking for.
\end{proofof}

\end{document}